\documentclass[11pt]{amsart}

\title[Spectral Theory and Descriptive Combinatorics]{Spectral Theory and Descriptive Combinatorics for Non-Regular Borel pmp Graphs}

\author[C.~Higgins]{Cecelia Higgins}
\address{Department of Mathematics, Rutgers University, 110 Frelinghuysen Road, Piscataway, NJ 08854-8019, USA}
\email{ch799@math.rutgers.edu}

\author[P.~Spaas]{Pieter Spaas}
\email{pspaas@wlu.ca}

\author[S.~Torelli]{Sofia Torelli}
\address{Department of Mathematics, University of California, Berkeley, 970 Evans Hall, Berkeley, CA 94720-3840, USA}
\email{storelli@berkeley.edu}

\usepackage[margin=1in]{geometry}

\usepackage[T1]{fontenc}

\usepackage{amssymb,
enumitem,
mathrsfs,
mathtools, 
tikz,
upgreek,
verbatim,
hyphenat,
url,
todonotes,
thm-restate,
microtype}

\makeatletter
\renewcommand\thmt@autorefsetup{\@xa\def\csname\thmt@envname autorefname\@xa\endcsname\@xa{\thmt@thmname}}
\makeatother

\newcommand{\C}{\mathbb{C}}
\newcommand{\N}{\mathbb{N}}

\newcommand{\R}{\mathbb{R}}

\newcommand{\eps}{\varepsilon}
\newcommand{\G}{\mathcal{G}}
\newcommand{\calN}{\mathcal{N}}
\newcommand{\ap}{\text{ap}}
\newcommand{\empt}{\emptyset}

\newcommand{\abs}[1]{\left| #1 \right|}
\newcommand{\norm}[1]{\left\| #1 \right\|}

\newcommand{\ceil}[1]{\left\lceil #1 \right\rceil}
\newcommand{\ang}[1]{\left\langle #1 \right\rangle}
\newcommand{\st}[1]{\left\{ #1 \right\}}
\newcommand{\en}[1]{\st{1,\dots,#1}}
\newcommand{\gmu}[1]{\gamma_\mu( #1 )}

\newcommand{\out}{\mathrm{out}}
\newcommand{\ins}{\mathrm{in}}
\newcommand{\dmu}{\ d\mu}
\newcommand{\sub}{\subseteq}
\newcommand{\ind}{\mathbf 1}
\newcommand{\sgn}{\text{sgn}}

\newenvironment{enumerate-(i)}{\begin{enumerate}[label={\upshape (\roman*)}, leftmargin=2pc]}{\end{enumerate}}
\newenvironment{enumerate-(1)}{\begin{enumerate}[label={\upshape (\arabic*)}, leftmargin=2pc]}{\end{enumerate}}

\theoremstyle{plain}

\newtheorem{thm}{Theorem}[section]
\newtheorem{cor}[thm]{Corollary}
\newtheorem{prop}[thm]{Proposition}
\newtheorem{lem}[thm]{Lemma}

\newtheorem{claim}[thm]{Claim}

\theoremstyle{definition}

\theoremstyle{remark}

\usepackage[colorlinks=true, linkcolor=cyan, citecolor=cyan, hypertexnames=false]{hyperref}

\numberwithin{equation}{section}

\begin{document}

\begin{abstract}
We continue the systematic study initiated in \cite{HST2026} of the relationship between the spectral theory and descriptive combinatorics of bounded-degree Borel pmp graphs, particularly in the non-regular setting.
We first adapt a classical result of Nikiforov \cite{Nikiforov2007} to give a new spectral lower bound on the approximate measurable chromatic number. For certain non-regular graphs, this bound improves on the approximate measurable version of Hoffman's classical bound established in \cite{HST2026}. 
Next, we strengthen a characterization in \cite{HST2026} of approximate measurable bipartiteness in terms of the spectrum of the adjacency operator by generalizing to the case of non-regular graphs.
Finally, we introduce the measurable domination number and adapt several classical results of Lu, Liu, and Tian \cite{LLT05} to give a bound on the spectral gap of the Laplacian operator in terms of the measurable domination number.
\end{abstract}

\maketitle

\section{Introduction}

In this paper, we study applications of spectral graph theory to descriptive combinatorics, particularly in the setting of non-regular graphs. This continues the line of exploration initiated in \cite{HST2026}. 

In descriptive combinatorics, a major goal is to develop tools for constructively solving graph theoretic problems such as coloring and matching on (certain) infinite graphs. Although many classical graph theory results can be solved on infinite graphs, the resulting solutions are often non-constructive, obtained by applying, for instance, the axiom of choice. Thus, in descriptive combinatorics, one typically imposes further definability or measurability constraints on the solutions. For modern surveys of the area, see, for instance, \cite{KechrisMarks2020, Pikhurko2021}.

In this paper, we study the descriptive combinatorics of a certain class of definable graphs. Specifically, given a standard probability space $(X,\mu)$, we are interested in bounded-degree Borel graphs that are probability-measure-preserving (or pmp) on $X$; see Section~\ref{sec: prelim} for definitions. This is a natural class of graphs that has been the subject of considerable attention in recent research; see, for instance, the survey \cite[Chapters~6--9 and~13]{KechrisMarks2020}, as well as the recent papers \cite{timar2019,thornton2022,bks2022,cmgt2025}.

Our techniques draw inspiration from classical spectral graph theory, in which matrices, such as the adjacency matrix and Laplacian matrix, are associated to finite graphs. Then, the eigenvalues of these matrices provide information about combinatorial properties of the corresponding graphs. For an introduction to spectral graph theory, we refer to the survey \cite{Spielman2025}.

Among the many successes of the classical theory are the various spectral bounds on the chromatic number. One particular lower bound due to Nikiforov \cite{Nikiforov2007}, which relates the chromatic number of a graph $G$ to the maximum eigenvalues of its adjacency and Laplacian matrices, denoted by $\mu_1(A)$ and $\mu_1(L)$, respectively, shows that
    \[\chi(G)\ge \ceil{
    1+\frac{\mu_1(A)}{\mu_1(L)-\mu_1(A)}
    }.\]
In the regular case, i.e., in the case that all vertices of $G$ have the same degree, this bound reduces to the well-known spectral bound of Hoffman \cite{Hoffman1970}. However, Nikiforov's bound is known to improve upon Hoffman's for certain non-regular graphs, while they are in general incomparable (see \cite{Nikiforov2007}). 

In Section \ref{sec: main theorem}, we generalize this result of Nikiforov to the setting of bounded-degree Borel pmp graphs. Specifically, let $\chi_\mu^\ap(\G)$ denote the approximate chromatic number of a graph $\G$, and let $M(T)$ denote the maximum spectral value of a bounded self-adjoint operator $T$ (see Section~\ref{sec: prelim} for precise definitions). Our first main result is the following.

\begin{restatable}[]{main}{maintheorem}\label{thm: approx m-nik}
    Let $\G$ be a bounded-degree Borel pmp graph on a standard probability space \((X,\mu)\) with adjacency operator $T_\G$ and Laplacian operator \(L_\G\). 
    Then
    \[\chi_\mu^{\ap}(\G)\ge \ceil{1+ \frac{M(T_\G)} {M(L_\G)-M(T_\G)}}.\]
\end{restatable}

Analogously to the classical setting, when $\G$ is regular, this bound reduces to the approximate measurable version of Hoffman's bound established in \cite[Theorem~6.8]{HST2026}. However, as mentioned earlier, there are non-regular graphs for which the approximate measurable version of Nikiforov's bound improves on that of Hoffman's bound.

Our proof of Theorem~\ref{thm: approx m-nik} follows in part from a more general result, Theorem~\ref{thm: meas Kol}, which relates the maximum spectral values of certain bounded self-adjoint operators. This theorem is inspired by a general result of Kolotilina \cite{Kolotilina2010} for matrices. Using Theorem~\ref{thm: meas Kol}, we are further able to state a spectral condition characterizing when the \emph{measurable} chromatic number $\chi_{\mu}(\G)$ is equal to the lower bound from Theorem~\ref{thm: approx m-nik}; see Corollary~\ref{cor: meas char}.

In Section~\ref{sec: bipartiteness}, we study approximate measurable bipartiteness. Concretely, the following is our next main result. 
\begin{restatable}[]{main}{bipartiteness}\label{thm: bipartite}
    Let $\G$ be an ergodic, bounded-degree Borel pmp graph on a standard probability space $(X, \mu)$, whose adjacency operator $T_\G$ has spectral gap. If \(-M(T_\G)\) belongs to the spectrum of $T_\G$, then $\G$ is approximately measurably bipartite.
\end{restatable}
This result is a generalization of \cite[Theorem~4.9]{HST2026} to the non-regular case. In the proof, we use the ergodicity of $\G$ to conclude that, up to scalar multiplication, the only eigenvector for $M(T_\G)$ is a strictly positive vector $h$. Then, taking a suitable approximate eigenvector $f$ for $-M(T_\G)$, we use the spectral gap assumption to conclude that $\vert f \vert$ is, modulo a small perturbation, a scalar multiple of $h$ and hence nowhere zero. Finally, we use the signs of $f$ to determine an approximate measurable bipartition of $\G$.

We remark that the converse of Theorem~\ref{thm: bipartite} is already known in full generality. That is, if a bounded-degree Borel pmp graph $\G$ is approximately measurably bipartite, then its spectrum is symmetric about $0$ \cite[Theorem~4.4]{HST2026}.

Finally, we consider the spectral gap of the Laplacian operator $L_\G$, which is given by
$$m_0(L_\G) = \inf_{f \in L^2_0(X), f \neq 0} \frac{\ang{L_\G f, f}}{\norm{f}^2}.$$
Note that when $\G$ is regular, one can equivalently define the spectral gap in terms of the adjacency operator; see, for instance, \cite{HST2026}. However, when $\G$ is non-regular, these operators can behave quite differently in terms of spectral gap. In many situations, the spectral gap of $L_\G$ is more natural, and it has already found applications to the descriptive combinatorics of $\G$ (see, for instance, \cite[Section~7]{HST2026}). It is therefore of interest to give quantitative bounds on this parameter, especially in the non-regular case.

This is the content of our last main result below, which establishes an upper bound on the Laplacian spectral gap in terms of the \textit{(measurable) domination number} of the graph. Recall that, for a finite graph $G$, a dominating set for $G$ is a set $Y \sub V(G)$ such that, for each vertex $x \in V(G)$, either $x \in Y$ or $x$ has a neighbor in $Y$. The domination number of $G$ is then the minimum cardinality of a dominating set. In Section~\ref{sec: domination}, we introduce the measurable domination number $\gamma_{\mu}(\G)$ as the infimum over the measures of (measurable) dominating sets. We then prove some relationships between $\gamma_{\mu}(\G)$ and the spectrum of $L_\G$. In particular, we provide the following upper bound on the size of the spectral gap of $L_\G$ in terms of $\gamma_{\mu}(\G)$. 

\begin{restatable}[]{main}{domination}\label{thm: domination}
    Let $\G$ be an ergodic Borel pmp graph on a standard probability space $(X, \mu)$ with maximum degree $\Delta$ and measurable domination number $\gamma_\mu$. Then
        $$m_0(L_\G) \leq \frac{\Delta(1 - 2\gamma_\mu) + \gamma_{\mu}}{\gamma_{\mu}(1 - \gamma_{\mu})}.$$
\end{restatable}

Our proof takes inspiration from the work of Lu, Liu, and Tian \cite{LLT05}, who prove a similar bound in the classical setting.

\subsection*{Acknowledgments}
S.T.\ was a participant in the 2026 DIMACS REU and was supported by NSF grants CCF-2447342 and DMS-2348819. The results of Sections~\ref{sec: main theorem} and~\ref{sec: bipartiteness} were obtained during the DIMACS program. P.S.\ was partially supported by the European Union via an ERC grant (QInteract, Grant No. 101078107). 

We thank Filippo Calderoni for his support throughout the REU and for the many helpful comments on early drafts of the paper.

\subsection*{A.I.\ disclosure}
No A.I.\ was used in the preparation of this document.

\section{Preliminaries}\label{sec: prelim}

\subsection{Graphs and colorings}
We use largely the same notations and conventions as in \cite{HST2026}, but we (re-)introduce the relevant ones here for the reader's convenience.

For the remainder of the paper, $(X,\mu)$ will denote a standard probability space unless specified otherwise. For notational convenience, given a subset $Y\sub X$, we will usually write $Y^c$ instead of $X\setminus Y$ when the ambient space is clear from context. 

A \emph{Borel graph} on $X$ is a graph $\G$ whose edge set $E(\G)$ is Borel in $X^2$. We write \(x\sim y\) when $x$ and $y$ are adjacent in $\G$; we also say $x$ is a \emph{neighbor} of $y$, and vice versa. In this paper, all graphs are undirected and contain no self-loops, i.e., $\sim$ is an irreflexive and symmetric relation.
The \emph{degree} of a vertex \(x\in X\), denoted by \(\deg(x)\), is the cardinality of its \emph{neighborhood} \(\calN(x) =\st{y\in X : y\sim x }\).
For a set $A\sub X$, we denote by \(\calN (A)\) the \emph{neighborhood} of $A$, i.e.,
\[\calN(A) = \bigcup_{x\in A}\calN(x)=\st{y \in X : y\text{ has a neighbor in }A}.\]
For any $Y\sub X$, let $\deg_Y(v)$ denote the degree of $v$ in $Y$, i.e., $\deg_Y(v)=\abs{\calN (v) \cap Y}$.

Since we work throughout in the measurable setting, we always implicitly discard sets of measure zero.
For instance, we say that $\G$ has \emph{bounded degree} if there is \(d\in \mathbb N\) such that \(\deg(x)\le d\) for all $x$ in a conull set $A\sub X$. When $\G$ has bounded degree, we define the \emph{maximum degree} $\Delta(\G)$ of $\G$ to be the maximum $d$ for which the set \(\{x\in X:\deg(x)=d\}\) has positive measure.
Similarly, $\G$ is $d$\emph{-regular} if there is \(d\in \mathbb N\) such that \(\deg(x) = d\) for all $x$ in a conull set. Finally, we note that, given a subset $A\sub X$, we may consider the \emph{subgraph $\G[A]$ of $\G$ induced by $A$}, i.e., the graph on $A$ whose edge relation $E(\G[A])$ is defined by
\[
E(\G[A]) = E(\G) \cap A^2.
\]

The set \(A \sub X\) is ($\G$-)\emph{independent} if \(x\not\sim y\) whenever \(x,y\in A\) are distinct. Let \(k\in \mathbb N\); a map \(c\colon X\to \en{k}\) is a $k$\emph{-coloring} of $\G$ if the color class \(c^{-1}(i)\) is independent for each \(i\in \en{k}\). The \emph{chromatic number} of $\G$, denoted by \(\chi(\G)\), is the least $k$ for which there is a $k$-coloring of $\G$.

The \emph{Borel chromatic number} of $\G$, denoted by $\chi_B(\G)$, is the least $k\in \N$ for which there is a Borel $k$-coloring of $\G$, i.e., a $k$-coloring whose color classes are Borel sets in $X$ (if such a $k$ exists).
Similarly, the \emph{measurable chromatic number} of $\G$, denoted by \(\chi_\mu(\G)\), is the least $k$ for which there is a measurable $k$-coloring of $\G$ (if such a $k$ exists).
The \emph{approximate measurable chromatic number} of $\G$, denoted by \(\chi_\mu^{\ap}(\G)\), is the least $k$ such that for any \(\eps>0\), there exists a Borel set \(A\sub X\) with \(\mu(A)>1-\eps\) for which the induced subgraph \(\G [A]\) has a measurable $k$-coloring (if such a $k$ exists).
Note that in general \(\chi(\G) \leq \chi_{\mu}(\G)\leq \chi_B(\G)\) and \(\chi_{\mu}^{\ap}(\G) \leq \chi_{\mu}(\G)\).

We write $E_\G$ for the \emph{connectedness relation} of $\G$, i.e., the equivalence relation whose classes are the
connected components of $\G$. When $\G$ has bounded degree, the relation $E_\G$ is a countable Borel equivalence relation. Given a set $A\sub X$, we denote by \([A]_{E_\G }\) the \emph{saturation} of $A$, i.e., the union of all components in $\G$ containing vertices in $A$. We say $A\sub X$ is \emph{invariant} if $A=[A]_{E_\G }$.

By considering the connectedness relation, we can make sense of various properties for $\G$ from the theory of Borel equivalence relations. For instance, we say $\G$ is \emph{ergodic} if $E_{\G}$ is ergodic, i.e., if every $E_\G$-invariant set has measure either $0$ or $1$. 
Additionally, we say that the graph $\G$ is \emph{($\mu$-)probability measure preserving}, or \emph{pmp} for short, if $E_{\G}$ is pmp, i.e., whenever $\theta$ is a partial Borel bijection on $X$ with $\text{graph}(\theta) \sub E_{\G}$, then $\theta$ preserves $\mu$ in the sense that $\mu(A) = \mu(\theta(A))$ for any $\mu$-measurable set $A \sub X$.

Let $E$ be a countable Borel equivalence relation on $(X,\mu)$. A \textit{Borel transport} in $E$ is a Borel function $\varphi: E\to [0,\infty)$. Given a Borel transport $\varphi$ in $E$, we define the functions $\mathrm{out}_\varphi, \mathrm{in}_\varphi: X\to [0,\infty)$ by
    \[
        \mathrm{out}_\varphi(x) = \sum_{y \mathrel{E} x} \varphi(x,y), \quad \text{and}\quad 
        \mathrm{in}_\varphi(x) = \sum_{y \mathrel{E} x} \varphi(y,x).
    \]
The well-known \emph{mass transport principle} demonstrates that, for any Borel transport $\varphi$, 
    \begin{equation}\label{eq:masstransport}
    \int_X \mathrm{out}_\varphi(x) \dmu(x) = \int_X \mathrm{in}_\varphi(x) \dmu(x).
    \end{equation}
For this formulation, see for instance \cite[Proposition~5.3]{tserunyan2022}, as well as \cite[Section~2]{HST2026} where several useful corollaries are deduced. We remark that applications of mass transport are ubiquitous in descriptive set theory (see, for instance, \cite{KechrisMiller2004}) and other areas. The underlying ideas trace back to the seminal work of Feldman and Moore \cite{fm1977}, who introduced the left and right counting measures and proved that their equality is equivalent to $\mu$-invariance (see also the edge measure in Section~\ref{ssec:op}).

\subsection{Spectral theory for bounded operators}

In this paper, we study certain operators associated with bounded-degree Borel pmp graphs. In this subsection, we collect some basic facts from operator theory that we use throughout; most of these can be found in any standard reference, e.g., \cite{Conway2000}.

Let $H$ be a complex Hilbert space with inner product denoted by \(\ang{\cdot, \cdot}\). For any linear operator $T$ on $H$, its \emph{operator norm} is given by \(\norm{T}=\sup\st{\norm{Tf} : f\in H, \norm{f} = 1}\).
We denote by \(B(H)\) the set of bounded operators on $H$, i.e., the linear operators with finite operator norm. 
The \emph{adjoint} of $T\in B(H)$ is the (unique) operator $T^*\in B(H)$ such that \(\ang{Tf,g}=\ang{f, T^*g}\) for all $f,g\in H$. 

For any \(T\in B(H)\), the \emph{spectrum} of $T$ is the set
    \[\sigma(T) = \st{\lambda\in \C: T-\lambda I \text{ is not invertible}}.\]
It is well-known that the spectrum is nonempty and compact in $\C$. 
The \emph{discrete spectrum} of $T$ is the set of $\lambda \in \sigma(T)$ for which $T-\lambda I$ is not injective, i.e., for which there exists an eigenvector $f\in H$ such that \(Tf=\lambda f\).\footnote{Some references use the term \emph{point spectrum} to refer to all eigenvalues, reserving the term ``discrete spectrum'' to mean the eigenvalues with finite algebraic multiplicity. This distinction is not important for our purposes.}
The \emph{continuous spectrum} of $T$ is the set of $\lambda \in \sigma(T)$ for which $T-\lambda I$ is injective but not surjective and has range dense in $H$.
The \emph{residual spectrum} of $T$ is the set of $\lambda \in \sigma(T)$ that are in neither the continuous spectrum nor discrete spectrum.

Most operators discussed in this paper will be \emph{self-adjoint}, i.e., satisfying $T^*=T$. For a self-adjoint operator $T$, it is well-known that $\sigma(T)\sub \R$, and that $T$ is positive if and only if $\sigma(T)\sub [0,\infty)$. Furthermore, in this case, $\lambda\in \sigma(T)$ if and only if there exists a sequence $(f_n)_{n \in \N}$ of unit vectors in $H$ such that
    \[\norm{Tf_n - \lambda f_n}\to 0.\]
In this case, we say that $\lambda$ is an \emph{approximate eigenvalue} of $T$ and that \((f_n)_{n\in \N}\) is a sequence of unit \emph{approximate eigenvectors} of $T$ associated with $\lambda$. This includes the special case when $\lambda$ is a true eigenvalue of $T$, in which case we can take the constant sequence $f_n=f$, where $f$ is an eigenvector for $\lambda$. Finally, when $T$ is self-adjoint, we denote by \(M(T)=\max\sigma(T)\) the maximum spectral value of $T$ and by \(m(T)=\min \sigma(T)\) the minimum spectral value of $T$. It is well-known that
\[
    M(T) = \sup_{f\in H,f\neq0} \frac{\ang{Tf,f}}{\ang{f,f}}, \quad \text{and} \quad m(T) = \inf_{f\in H,f\neq0} \frac{\ang{Tf,f}}{\ang{f,f}}.
\]

Finally, we establish some notation regarding the \emph{block decomposition} \((T_{ij})_{i,j=1}^n\) of $T\in B(H)$ associated with a direct sum decomposition \(\bigoplus_{i=1}^n H_i\) of $H$. That is, \(T_{ij}\in B(H_j, H_i)\) for every \(i,j \in \en{n}\) and given \(x = \sum_{i=1}^n x_i\) with \(x_i\in H_i\), we have
    \[Tx = \sum_{i=1}^n\sum_{j=1}^n T_{ij}(x_j).\]
Note that this is just matrix multiplication with respect to the decomposition of $H$:
    \[Tx = 
    \begin{pmatrix}
        T_{11}&T_{12}&\cdots &T_{1n}\\[4pt]
        T_{21}&T_{22}&\cdots &T_{2n}\\[4pt]
        \vdots &\vdots &\ddots &\vdots \\[4pt]
        T_{n1}&T_{n2}&\cdots &T_{nn}\\[4pt]
    \end{pmatrix}
    \begin{pmatrix} 
        x_{1}\\[4pt]x_{2}\\[4pt]\vdots \\[4pt]x_{n}\\[4pt]
    \end{pmatrix} = 
    \begin{pmatrix}
        \sum_{j=1}^{n}T_{1j}(x_{j})\\[4pt]
        \sum_{j=1}^{n}T_{2j}(x_{j})\\[4pt]
        \vdots \\[4pt]
        \sum_{j=1}^{n}T_{nj}(x_{j})\\[4pt]
    \end{pmatrix}. \]
It is easy to check that if $T$ is self-adjoint, then \(T_{ij}=T_{ji}^*\) for every \(i, j\in \en{n}\). We also have the following useful result.

\begin{prop}[{\cite[Proposition~2.7]{HST2026}}]
\label{prop: block decomp T}
Let \(T\in B(H)\) be a self-adjoint operator. Then, for any $i,j\in \en{n}$:
\begin{enumerate-(1)}
    \item \(M(T_{ii})\le M(T)\) and \(m(T_{ii})\ge m(T)\), i.e., \(\sigma(T_{ii})\sub [m(T), M(T)]\).\label{line: block decomp Mm}
    \item \(\norm{T_{ij}} \le \norm{T}\).\label{line: block decomp norms}
\end{enumerate-(1)}
\end{prop}

To extend results on the measurable chromatic number to the approximate measurable chromatic number, we will also need the following result. The proof generalizes an argument presented in \cite[Theorem~6.8]{HST2026}. 

\begin{lem}
\label{lem: approx M}
Let \(T\colon L^2(X)\to L^2(X)\) be a bounded self-adjoint operator on $(X,\mu)$. For any $\eps>0$, let \(X=A^\eps\sqcup B^\eps\) be a measurable partition of $X$ such that \(\mu(A^\eps) > 1-\eps\). Let
\[T
=\begin{pmatrix}T^\eps_{A} & T^\eps_{12}\\[4pt]T^\eps_{21} & T^\eps_{B}\end{pmatrix}\] 
be the block decomposition of $T$ with respect to the direct sum decomposition \(L^2(X)= {L^2(A^\eps)} \oplus {L^2(B^\eps)}\). Then
\[\lim_{\eps\to 0} M(T_A^{\eps}) = M(T).\]
\end{lem}

\begin{proof}
Fix \(\eps>0\). Let
    \[\eps_0=\frac \eps{3\norm{T}+1}.\]
There exists a unit vector \(f\in L^2(X)\) such that \(\ang{T f, f}>M(T)-\eps_0\).
By the absolute continuity of the integral, there exists \(\delta>0\) such that $\mu(B)<\delta $ implies
    \[\int_B \abs{f(x)}^2  \dmu(x) < \eps_0^2\]
for any measurable \(B\sub X\).
Let \(\eps'<\delta\), and let \(f=f_1+f_2\) be the unique decomposition of $f$ such that \(f_1\in L^2(A^{\eps'})\) and \(f_2\in L^2(B^{\eps'})\).
Using Proposition~\ref{prop: block decomp T}\ref{line: block decomp norms}, we get

\begin{align*}
    \norm{\ang{T f,f} -\ang{T_A^{\eps'}f_1, f_1}} \le 
        \norm{\ang{T_{12}^{\eps'} f_2,f_1}} + 
        \norm{\ang{T_{21}^{\eps'} f_1,f_2}} + 
        \norm{\ang{T_{B}^{\eps'} f_2,f_2}}\\[4pt]
    \le \norm{T_{12}^{\eps'}} \norm{f_2}\norm{f_1} + 
        \norm{T_{21}^{\eps'}} \norm{f_1} \norm{f_2} + 
        \norm{T_{B}^{\eps'} } \norm{f_2}^2 \\[4pt]
    \le 3 \norm{T} \norm{f_2} \le 3\norm{T}\eps_0.
\end{align*}
where the last inequality follows from the fact that $\eps'<\delta$ and
    \[\|f_2\|=\left(\int_{B^{\eps'}} \abs{f}^2 \dmu\right)^\frac12 <{\eps_0}.\]
Hence,
    \begin{align*}
        M(T_A^{\eps'}) &\ge \ang{T^{\eps'}_Af_1, f_1}
        \ge  \ang{Tf, f} -3\norm{T}\eps_0 \\
        &\ge M(T) - \left(3\norm{T}+1\right) \eps_0 =M(T)-\eps.
    \end{align*}
We know from Proposition~\ref{prop: block decomp T}\ref{line: block decomp norms} that \(M(T_A^{\eps})\le M(T)\) for any $\eps$. 
We conclude that indeed, \(\lim_{\eps\to0} M(T_A^{\eps}) = M(T)\).
\end{proof}

\subsection{Operators for bounded-degree graphs}\label{ssec:op}
Given a graph $\G$, we consider three key operators that store information about the structure of $\G$. These definitions extend the notions of the adjacency, degree, and Laplacian matrices for finite graphs. 
Moreover, when $\G$ is a bounded-degree Borel pmp graph, it is known that all three operators are bounded and self-adjoint; see, for instance, \cite{HST2026}, where various basic properties of these operators are also collected. 

For the remainder of this subsection, fix a bounded-degree Borel pmp graph $\G$ on $(X, \mu)$. The \emph{adjacency operator} \(T_\G \colon L^2(X)\to L^2(X)\) is defined by
    \[\forall f\in L^2(X),\ \forall x\in X\quad
    \left(T_\G  f\right)(x) = \sum_{y\sim x} f(y).\]
The \emph{degree operator} \(D_\G \colon L^2(X)\to L^2(X)\) is the multiplication operator defined by
    \[\forall f\in L^2(X),\ \forall x\in X\quad
    \left(D_\G f\right)(x)=\deg(x)\cdot f(x).\]
Finally, the \emph{Laplacian operator} \(L_\G \colon L^2(X)\to L^2(X)\) is defined by \(L_\G =D_\G -T_\G\). Observe that
    \[(L_\G f)(x)
    =\deg(x)f(x) -\sum_{y\sim x}f(y)
    =\sum_{y\sim x}
    \left(f(x) - f(y)\right).\]
In particular, one can see immediately that $L_\G \ind = 0$, i.e., the all-ones vector is an eigenvector of $L_\G$ associated with eigenvalue $0$. Note that, when $\G$ is $d$-regular, then $L_\G = dI - T_\G$.

It is well-known that the Laplacian matrix of a finite graph is positive semidefinite. Similarly, the Laplacian operator is always positive; see, for instance, \cite[Lemma~3.11]{HST2026}. In fact, one can show that for $f \in L^2(X)$, we have
    $$\ang{L_\G f, f} = \frac{1}{2} \int_X \sum_{y \sim x} \vert f(x) - f(y) \vert^2 \dmu(x).$$
In classical spectral graph theory, it is also well-known that the maximum eigenvalue of the Laplacian is at least the maximum degree plus one. We next observe that this result may be extended as follows.

\begin{lem}\label{lem: M(L) degree bound}
    Let $\G$ be a bounded-degree Borel pmp graph on $(X,\mu)$ with maximum degree $\Delta$. Then $M(L_\G) \geq \Delta+1$.
\end{lem}
\begin{proof}
The proof is similar to that of \cite[Lemma~3.8(2)]{HST2026}. First, combining \cite[Proposition~4.6]{KST1999} and (the proof of) \cite[Corollary~2.11]{HST2026}, one can find a set $A\sub X$ such that $\mu(A)>0$, every $x\in A$ satisfies $\deg(x) = \Delta$, different points in $A$ are at distance at least 4 from each other, and $\mu(\calN(A)) = \Delta\mu(A)$.

Now define the function $f\in L^2(X)$ as follows:
\begin{align*}
f(x) = \begin{cases}
    \Delta, &\text{ if } x\in A,\\
    -1, &\text{ if } x\in \calN(A),\\
    0, &\text{ otherwise}.
\end{cases}
\end{align*}
Then we can compute 
$$
\langle f,f\rangle = \int_X \abs{f(x)}^2\dmu(x) = \Delta^2\mu(A) + \mu(\calN(A)) = \Delta(\Delta+1)\mu(A),
$$
and
\begin{align*}
\langle L_\G f, f\rangle &= \frac 12 \int_X \sum_{y\sim x} \abs{f(x) - f(y)}^2 \dmu(x)\\
&\geq \frac 12 \int_A \sum_{y\sim x} \abs{f(x) - f(y)}^2 \dmu(x) + \frac 12 \int_{\calN(A)} \sum_{y\sim x} \abs{f(x) - f(y)}^2 \dmu(x)\\
&\geq \frac 12 \Delta(\Delta+1)^2 \mu(A) + \frac 12 (\Delta+1)^2 \mu(\calN(A))\\
&=\Delta(\Delta+1)^2 \mu(A).
\end{align*}
We conclude that 
$$
M(L_\G) \geq \frac{\langle L_\G f, f\rangle}{\langle f,f\rangle} \geq \frac{\Delta(\Delta+1)^2 \mu(A)}{\Delta(\Delta+1)\mu(A)} = \Delta+1,
$$
finishing the proof.
\end{proof}

Finally, we prove a few lemmas that will be useful later. 

\begin{lem}\label{lem: ergodic lap. quadratic form = 0}
    Suppose $\G$ is an ergodic bounded-degree Borel pmp graph on $(X,\mu)$ and \(f\in L^2(X)\) satisfies \(\ang{L_\G  f,f} = 0\). Then $f$ is constant on an invariant conull set.
\end{lem}

\begin{proof}
    Assume
    \[\ang{L_\G f, f} = \frac 12 
        \int_X \sum_{y\sim x} \abs{f(x) - f(y)}^2 \dmu(x) = 0.\]
    Then there exists some conull set \(A\sub X\) on which the integrand is $0$, i.e., \(f(x)=f(y)\) for any \(x\in A\) and any $y$ adjacent to it.
    Let \(B = X\backslash A\) be its complement, i.e.,
    \[B = \st{x\in X:
    \exists y \sim x \quad f(x)\ne f(y) }.\]
    Then $B$ is a null set, and since $E_\G$ is a countable Borel pmp equivalence relation, it is well-known that also its saturation $\left[B\right]_{E_\G }$ is a null set. Since $f$ is constant on connected components contained in the complement of $[B]_{E_\G}$ by construction, the result follows from ergodicity.
\end{proof}

We next observe that in the block decomposition of the adjacency operator with respect to the color classes of a (measurable) coloring, all diagonal blocks are necessarily the zero operator. 

\begin{lem}\label{lem: zero-diagonal}
    Let $\G$ be a bounded-degree Borel pmp graph. Suppose \(c\colon X\to\{1,\dots,k\}\) is a measurable $k$-coloring of $\G$ with color classes \(X_i=c^{-1}(i)\).
    Consider the block decomposition \((T_{ij})_{i,j=1}^k\) of \(T_\G\) with respect to the direct sum decomposition
    \[L^2(X) = \bigoplus_{i=1}^k L^2(X_i).\]
    Then \(T_{ii}=0\) for all \(i\in \en{k}\).
\end{lem}

\begin{proof}
    For any $i\in \en{k}$, the block diagonal operator \(T_{ii}\) is the operator on $L^2(X_i)$ such that 
        \[(T_{ii}f)(x) = \sum_{\substack{y \sim x\\ y\in X_i}} f(y)\] 
    for each $f\in L^2(X_i)$ and $x\in X_i$.
    By definition of the color classes, if $y$ is adjacent to a vertex $x \in X_i$, then $y \notin X_i$. So for any \(f\in L^2(X_i)\) and any \(x\in X_i\),
        \[\left(T_{ii} f\right)(x)
        =\sum_{\substack{y \sim x\\ y\in X_i}} f(y)
        =0. \qedhere\]
\end{proof}

Next, we observe the following fact, which was proved in \cite[Corollary~3.16]{HST2026} for regular graphs.

\begin{lem}\label{lem: -M and abs eig'fns}
    If \(-M(T_\G) \in \sigma(T_\G)\) and \((f_n)_{n\in\N}\) is a sequence of unit approximate eigenvectors of $T_\G$ associated with $-M(T_\G)$, then \(\norm{T_\G\abs{f_n} - M(T_\G)\abs{f_n}} \to 0\).
\end{lem}

\begin{proof}
    Since \(\norm{T_\G f_n + M(T_\G)f_n} \to 0\), it follows that \(\ang{T_\G f_n, f_n} \to -M(T_\G)\). Then,
        \[\abs{\ang{T_\G f_n, f_n}}
        \le \ang{T_\G \abs{f_n}, \abs{f_n}}
        \le M(T_\G)\]
    for all $n\in \N$. However, in the limit \(\abs{\ang{T_\G f_n, f_n}}\to M(T_\G)\), forcing \(\abs{\ang{T_\G \abs{f_n}, \abs{f_n}}}\to M(T_\G)\) and
    \begin{align*}
        \norm{T_\G\abs{f_n} - M(T_\G)\abs{f_n}}^2
        &=
        \ang{T_\G\abs{f_n} - M(T_\G)\abs{f_n},T_\G\abs{f_n} - M(T_\G)\abs{f_n}}\\[4pt]
        &\le \norm{T_\G}^2
        - 2M(T_\G)\ang{T_\G\abs{f_n}, \abs{f_n}} 
        + M(T_\G)^2\\[4pt]
        &=2M(T_\G)\left(M(T_\G)
        - \ang{T_\G\abs{f_n}, \abs{f_n}}\right)\\
        &\to 0.
    \end{align*}
    Thus \(\norm{T_\G\abs{f_n} - M(T_\G)\abs{f_n}} \to 0\) as desired.
\end{proof}

Finally, we record an easy observation regarding edge measures: given a bounded-degree Borel pmp graph $\G$ on $(X,\mu)$, the \emph{edge measure} $\nu$ is defined such that, for each measurable $A \sub E(\G)$,
    \[\nu(A)=\int_X \abs{A_x} \dmu(x),\]
where \(A_x=\left\{ y\in X: (x,y)\in A \right\}\). Note that by a straightforward application of the mass transport principle~\eqref{eq:masstransport}, we could equivalently define the edge measure as
    \[\nu(A)=\int_X \abs{A^x} \dmu(x),\]
where \(A^x=\left\{ y\in X: (y,x)\in A \right\}\). In particular, given a subset $Y\sub X$, the measure of its \textit{boundary edges} $E_Y = E(\G) \cap (Y\times Y^c)$ satisfies
\[
\nu(E_Y) = \int_Y \deg_{Y^c}(x) \dmu(x) = \int_{Y^c} \deg_{Y}(x) \dmu(x).
\]

\begin{lem}\label{lem: edge density representor}
    Let $\G$ be a bounded-degree Borel pmp graph on $(X,\mu)$. Let $Y\sub X$ be a Borel subset.
    Then $f = \ind_Y-\mu(Y)\mathbf 1 \in L^2_0(X)$ satisfies
        \[\frac{\ang{L_\G f, f}} {\norm{f}^2}\
        = \frac {\nu(E_Y)}{\mu(Y)(1-\mu(Y))}.\]
\end{lem}

\begin{proof}
    First, we can compute
    \begin{align*}
        \int_X f(x) \dmu(x)
    &=
    \int_Y (1-\mu(Y)) \dmu(x)
    +\int_{Y^c} -\mu(Y) \dmu(x)\\
    &=\mu(Y)(1-\mu(Y))
    -(1-\mu(Y))\mu(Y)
    =0,
    \end{align*}
    and so $f$ is indeed orthogonal to $\ind$. Since $L_\G \mathbf 1=0$, we deduce that $L_\G f = L_\G \ind_Y$, and
        \[L_\G f(x) = L_\G
        \ind_Y(x)
        =\sum_{y\sim x}
        \left(\ind_Y(x) - \ind_Y(y)\right)
        =\begin{cases}
        \deg_{Y^c}(x)&x\in Y\\
        -\deg_Y(x)&x\in Y^c
        \end{cases}.\]
    We thus get
    \begin{align*}
        \ang{L_\G f,  f}
        &= \ang{L_\G \ind_Y,  \ind_Y} 
        -\mu(Y) \ang{L_\G \ind_Y,  \ind}
        = \ang{L_\G  \ind_Y,  \ind_Y} 
        -\mu(Y) \ang{\ind_Y,   L_\G \ind}\\
        &= \int_X L_\G  \ind_Y(x)\ \overline{\ind_Y(x)} \dmu (x)
        = \int_Y L_\G  \ind_Y(x) \dmu (x)\\
        &= \int_Y \deg_{Y^c}(x) \dmu (x) =\nu(E_Y).
    \end{align*}
    Finally,
    \begin{align*}
        \norm{f}^2
        &=\int_X
        ({\ind_Y(x)}^2-2\mu(Y)\ind_Y(x)
        +\mu(Y)^2) \dmu(x)\\
        &=\mu(Y)-2\mu(Y)^2+\mu(Y)^2=\mu(Y)(1-\mu(Y)),
    \end{align*}
    as desired.
\end{proof}

\subsection{Spectral gap}

In this final preliminary subsection, we briefly discuss spectral gap. For regular graphs, one can easily define spectral gap using either the adjacency operator or, equivalently, the Laplacian operator; see, for instance, \cite{HST2026}. When $\G$ is not regular, the spectral gaps for these operators do not necessarily coincide, even for finite connected graphs. In fact, in the Borel setting, it is not clear to us whether $T_\G$ having spectral gap is equivalent to $L_\G$ having spectral gap. While the Laplacian operator seems to yield a more natural definition of spectral gap in many situations, we will introduce both the spectral gap of the adjacency operator and that of the Laplacian operator here, since we will use them both in this paper.

We will say an ergodic graph $\G$ has \emph{adjacency spectral gap} if $M(T_\G)$ is isolated in $\sigma(T_\G)$, i.e., if there exists $\eps>0$ such that \(\sigma(T_\G) \sub \left[m(T_\G), M(T_\G)-\eps\right] \cup \{M(T_\G)\}\).

Similarly, we will say an ergodic graph $\G$ has \emph{Laplacian spectral gap} if $m(L_\G) = 0$ is isolated in $\sigma(L_\G)$, i.e., if there exists $\eps>0$ such that \(\sigma(L_\G) \sub \{0\} \cup \left[\eps, M(L_\G)\right]\).
We will also write 
$$m_0(L_\G) = \inf_{f \in L^2_0(X), f \neq 0} \frac{\ang{L_\G f, f}}{\norm{f}^2}.$$
Note that by construction, when $\G$ has Laplacian spectral gap, $m_0(L_\G)$ is exactly the spectral gap, i.e., the distance between $0$ and the second-smallest spectral value.

\section{A Lower Bound on the Chromatic Number}\label{sec: main theorem}

Our main goal in this section is to prove the following new lower bound on the approximate measurable chromatic number.

\maintheorem*

We remark first that, when $\G$ is $d$-regular, the lower bound in Theorem~\ref{thm: approx m-nik} reduces to the Hoffman lower bound in \cite[Theorem~6.8]{HST2026}. Indeed, in this case, $M(T_\G) = d$ and $M(L_\G) = d - m(T_\G)$, so that 
$$1 - \frac{M(T_\G)}{m(T_\G)} = 1 - \frac{d}{m(T_\G)}$$
and
$$1 + \frac{M(T_\G)}{M(L_\G) - M(T_\G)} = 1 + \frac{d}{(d - m(T_\G)) - d} = 1 - \frac{d}{m(T_\G)}.$$

However, for non-regular graphs, the Nikiforov bound can be an improvement over the Hoffman bound, and in general the two bounds are incomparable (see \cite{Nikiforov2007}). It appears to be open whether there is a simple characterization of the graphs for which Nikiforov's bound is tighter than Hoffman's.

\subsection{Main Technical Theorem}

To obtain Theorem~\ref{thm: approx m-nik}, as well as various further results, we first adapt a general theorem of Kolotilina \cite{Kolotilina2010} for matrices to the setting of bounded self-adjoint operators. 

\begin{thm} \label{thm: meas Kol}
    Let $(X,\mu)$ be a measure space and $L\colon L^2(X)\to L^2(X)$ be a bounded self-adjoint operator. Suppose $X$ is measurably partitioned into $X=\bigsqcup_{i=1}^k X_i$.
    Let \(L=(L_{ij})_{i,j=1}^k\) be the block partition of $L$ according to the direct sum $L^2(X)=\bigoplus_{i=1}^k
    L^2(X_i)$.
    Suppose $D = Diag(L_{11},\dots,L_{kk})$ is the block diagonal part of $L$,  and $T = D-L$. Moreover, let $\ind_i$ be the indicator function on $X_i$ for each $i\in \en{k}$. Then:
    \begin{enumerate-(1)}
        \item For every \(\xi\in \left[-\frac 1{k-1}, 1\right],\)
            \[M(L)\ge M(D-\xi T ).\]
        \label{line: meas Kol 1}
        \item If \(\xi\in \left(-\frac 1{k-1}, 1\right)\), then equality holds in \ref{line: meas Kol 1} if and only if there are some $i\in\en{k}$ and some sequence of unit approximate eigenvectors $(h_n)_{n\in \N}$ of $L$ associated with $M(L)$ such that each $h_n$ is supported only on the $i$th component, i.e., $h_n=\ind_ih_n$ for all $n\in \N$.
        \label{line: meas Kol 2}
        \item If $\xi = -\frac 1{k-1}$ and $k\ge 3$, then equality holds in \ref{line: meas Kol 1} if and only if there exist a sequence $(h_n)_{n\in \N}$ of unit vectors in $L^2(X)$ and some $i\in \en{k}$ such that \(\left((\ind_i-\ind_j)h_n\right)_{n\in \N}\) is a sequence of approximate eigenvectors of \(L\) associated with $M(L)$ for any $j\in\en{k}\backslash\{i\}$.
        \label{line: meas Kol 3}
    \end{enumerate-(1)}
    
\end{thm}

\begin{proof}
    \ref{line: meas Kol 1}
    Let $f\in L^2(X)$ be a nonzero function. Fix $\alpha\in \R$. Then, for any $i\in \en{k}$, let $g^{(i)}=f+\alpha\ind_if$.
    Notice that
        \begin{align}
        \begin{split}\label{eq:Kol1}
        \sum_{i=1}^k \ang{Lg^{(i)}, g^{(i)}}
        &= \sum_{i=1}^k \ang{L(f+\alpha \ind_if), f+\alpha \ind_if}\\
        &= k \ang{ Lf, f}
        + \alpha \sum_{i=1}^k \ang{Lf, \ind_if}\\
        &\quad + \alpha \sum_{i=1}^k \ang{L\ind_if, f}
        + \alpha^2 \sum_{i=1}^k \ang{L\ind_if, \ind_if}\\
        & = \left(k + 2\alpha \right) \ang{Lf, f}
        + \alpha^2 \ang{Df,f}.
        \end{split}
        \end{align}
    The exact same computation without $L$ also gives
        \begin{equation}\label{eq:Kol2}
        \sum_{i=1}^k \norm{g^{(i)}}^2
        = \left(k + 2\alpha+\alpha^2 \right) \norm{f}^2.
        \end{equation}
    Denote
        \[K= \frac 1{k + 2\alpha+\alpha^2}
        \left[ \left(k+2\alpha\right)L + \alpha^2 D\right].\]
    Notice that if \(\xi=\frac{k + 2\alpha}{k + 2\alpha+\alpha^2}\),
        \[ K=\frac 1{k + 2\alpha+\alpha^2}
        \left[ -\left(k+2\alpha\right)T
        + \left(k + 2\alpha+\alpha^2\right)D \right]
        =D-\xi T\]
    and \(\xi\in \left[-\frac 1{k-1}, 1 \right]\) whenever \(\alpha\in \left[-k, 0\right]\).

    Now, pick $(f_n)_{n\in \N}$ to be a sequence of unit vectors in $L^2(X)$ such that for every $n\in \N$,
        \[\ang{Kf_n, f_n}  \ge M(K)-\frac 1n .\]
    Then, for each $n\in \N$ and $i\in \en{k}$, let $g_n^{(i)}=f_n+\alpha \ind_i{f_n}$. Using \eqref{eq:Kol1} and \eqref{eq:Kol2}, we can now compute
        \begin{equation}
        \begin{split}
            M(K)-\frac 1n  
            \le  \ang{Kf_n, f_n}
            &= \frac 1{k+2\alpha+\alpha^2}\ \left[ \left(k + 2\alpha \right) \ang{ Lf_n, f_n } + \alpha^2 \ang{Df_n,f_n}\right]\\[4pt]
            &= \frac 1{k+2\alpha+\alpha^2} \sum_{i=1}^k \ang{Lg_n^{(i)}, g_n^{(i)}} \\[4pt]
            &\le \frac {M(L)}{k+2\alpha+\alpha^2} \sum_{i=1}^k \norm{g_n^{(i)}}^2 
            =M(L)\norm{f_n}^2 =M(L). \label{eqn: multline}
        \end{split}
        \end{equation}
    Since $n$ can be made arbitrarily large, we get \(M(K)\le M(L)\), as desired.

    \ref{line: meas Kol 2}
    Consider the series of inequalities \eqref{eqn: multline} used to prove \ref{line: meas Kol 1}. If we assume $\lim_{n\to \infty} \left( M(K)-\frac 1n\right)=M(L)$, then necessarily
        \[\lim_{n\to\infty} \left(M(L) \norm{g_n^{(i)}}^2 - \ang{Lg_n^{(i)}, g_n^{(i)}}\right)=0\]
    for all $i\in \en{k}$. 
    Since $g_n^{(i)}=f_n+\alpha\ind_i {f_n}$ and $\norm{f_n}=1$ for all $i\in \en{k}$, the sequence of norms \( \left(\norm{g_n^{(i)}}\right)_{n\in \N} \) is bounded above, and thus we can further deduce that
        \[\lim_{n\to \infty} \norm{Lg_n^{(i)} -M(L)g_n^{(i)}} =0\]
    for all $i \in \en{k}$ (as in the proof of Lemma~\ref{lem: -M and abs eig'fns}). 
    Then, notice that $\sum_{i=1}^k g_n^{(i)}=(\alpha+k)f_n$, so
        \[(\alpha+k) \lim_{n\to\infty} \norm{Lf_n - M(L)f_n}
        \le \lim_{n\to\infty} \sum_{i=1}^k \norm{Lg_n^{(i)} - M(L)g_n^{(i)}} =0.\]
    Therefore, assuming $\alpha\ne -k$ (i.e., $\xi\ne -\frac 1{k-1}$), $(f_n)_{n\in \N}$ is a sequence of approximate eigenvectors of $L$ associated with $M(L)$. Then for each $i\in \en{k}$, notice that $\alpha \ind_i{f_n}=g_n^{(i)}-f_n$, so
        \[\alpha \lim_{n\to\infty} \norm{L\ind_i f_n-M(L)\ind_i f_n}
        \le \lim_{n\to\infty} \norm{Lg_n^{(i)}-M(L)g_n^{(i)}} + \norm{L{f_n}-M(L){f_n}}=0. \]
    Assuming $\alpha\ne 0$ (i.e., $\xi\ne 1$), we get 
    \(\lim_{n\to\infty} \norm{L\ind_i f_n-M(L)\ind_i f_n}=0\). 
    Seeing as 
    \(\sum_{i=1}^k \norm{\ind_i f_n} \ge \norm{f_n} = 1\),
    for any $n\in \N$, there is some $i\in \en{k}$ such that \(\norm{\ind_if_n}\ge \frac 1k\).
    Then, by a pigeonhole argument, for some fixed $i\in \en{k}$, there exists a subsequence 
    \(\left(\ind_i f_{n_j}\right)_{j \in \N}\) with norms bounded below, i.e. $\norm{\ind_i f_{n_j}} \geq \frac 1k$ for all $j\in \N$.
    Thus, without loss of generality, we can rescale this subsequence
        \[h_j = \frac 1{\norm{\ind_i f_{n_j}}}\ind_i f_{n_j}\]
    to achieve a sequence \(\left(h_j\right)_{j\in \N}\), supported on $X_i$, of unit approximate eigenvectors of $L$ associated with $M(L)$:
        \[\lim_{j\to\infty} 
        \norm{Lh_j - M(L)h_j} \leq k
        \lim_{j\to\infty} 
        \norm{L\ind_i f_{n_j}-M(L)\ind_i f_{n_j}} = 0. \]
    
    Conversely, assume there is a sequence of (unit) approximate eigenvectors $(h_n)_{n\in \N}$ of $L$ associated with $M(L)$ which is supported on $X_i$ for some $i\in \en{k}$. 
    Recall that $D$ is the block diagonal part of $L$, and $T = D-L$. Hence, $Dh_n$ is be supported on $X_i$, whereas $-Th_n$ is supported on $X\backslash X_i$. Thus, we can decompose the norm
        \[\norm{Lh_n-M(L)h_n}^2
        =\norm{Dh_n-M(L)h_n}^2 + \norm{-Th_n}^2.\]
    Therefore,
        \begin{align*}
        \lim_{n\to \infty} \norm{D h_n - M(L) h_n}
        &\le \lim_{n\to \infty}
        \norm{L h_n - M(L) h_n} = 0, \text{ and}\\
        \lim_{n\to \infty}
        \norm{T h_n}
        &\le \lim_{n\to \infty}
        \norm{L h_n - M(L) h_n} = 0
        \end{align*}
    We deduce that
        \[\lim_{n\to \infty} 
        \norm{(D- \xi T)h_n - M(L)h_n}
        \le \lim_{n\to \infty} 
        \norm{Dh_n - M(L)h_n}
        +\abs{\xi}
        \norm{ Th_n} = 0.\]
    In particular, $M(L)$ is a spectral value of $D-\xi T$, and thus \(M(L)\le M(D-\xi T)\), which together with the inequality in \ref{line: meas Kol 1} guarantees the equality \(M(D-\xi T)=M(L)\).

    \ref{line: meas Kol 3}
    Assume $M(D-\xi T)=M(L)$ for $\xi=\frac 1{k-1}$. Then, let $f_n$ be a sequence of approximate eigenvectors of $L$ associated with $M(L)$ and \(g_n^{(i)} = f_n +\alpha \ind_i f_n\) where $\alpha=-k$.
    From the proof of \ref{line: meas Kol 2} we know that $\left(g_n^{(i)}\right)_{n\in \N}$ is a sequence of approximate eigenvectors of $L$ associated with $M(L)$ for each $i \in \en{k}$, i.e.,
        \[\lim_{n\to \infty} \norm{Lg_n^{(i)} -M(L)g_n^{(i)}}=0.\]
    Then, for any $j\ne i$ and $n\in \N$ notice that
        \[\left(\ind_i-\ind_j\right){f_n} =\frac{g_n^{(j)}-g_n^{(i)}}{k},\]
    implying that \(\left(\left(\ind_i-\ind_j\right){f_n}\right)_{n\in \N}\) must also be a sequence of approximate eigenvectors of $L$ associated with $M(L)$. So $(f_n)$ satisfies the desired property from \ref{line: meas Kol 3}.

    Conversely, assume there exists a sequence of unit vectors $(h_n)_{n\in \N}$ and some $i\in \en{k}$ such that for any $j\ne i$, \(\left((\ind_i-\ind_j)h_n\right)_{n\in \N}\) is a sequence of approximate eigenvectors of $L$ associated with $M(L)$.
    Similar to the proof of \ref{line: meas Kol 2}, in order to prove $M\!\left(D+\frac 1{k-1} T\right)=M(L)$, it suffices to show $M(L)$ is in the spectrum of $D+\frac 1{k-1} T$. We will do so by proving $(h_n)_{n\in\N}$ is a sequence of approximate eigenvectors of \(D+\frac 1{k-1} T\) associated with $M(L)$, i.e.,
        \[\lim_{n\to \infty} \norm{\left(D+\frac 1{k-1} T\right){h_n} -M(L){h_n}}=0.\]
    To establish this, we note that it is sufficient to show the norm in the ${X_l}^{\text{th}}$ component goes to zero for each $l\in \en{k}$, i.e.,
        \[\lim_{n\to \infty} \norm{\left(D+\frac 1{k-1} T\right){h_n} -M(L){h_n}}_{X_l}=0,\]
    where \({\norm{f}_{X_l}}\) denotes \(\left(\int_{X_l} \abs{f}^2\ d\mu\right)^\frac 12\).
    
    For any $j \ne i$, we know that
    \begin{equation}
        \lim_{n\to \infty} \norm{L( \left(\ind_i - \ind_j\right) h_n) 
        -M(L) \left( \ind_i - \ind_j \right) h_n}=0.\label{eq: 1i-1j approx eigvecs}
    \end{equation}
    Considering just the block components in \eqref{eq: 1i-1j approx eigvecs} with output in $L^2(X_i)$, we get
    \[
        \lim_{n\to \infty} \norm{
             L_{ii} \ind_i h_n
            - L_{ij} \ind_j h_n
            -M(L)\ind_i{h_n}}=0
    \]
    for all $j\ne i$. Therefore
    \begin{align*}
        &\lim_{n\to \infty}\norm{\left(D+\frac 1{k-1}T\right)h_n-M(L)h_n}_{X_i}\\[4pt]
        &=\lim_{n\to \infty}
        \norm{L_{ii}\ind_ih_n - \frac 1{k-1} \sum_{j\ne i} L_{ij}\ind_jh_n -M(L)\ind_ih_n}\\[4pt]
        &\le\lim_{n\to \infty}
        \frac 1{k-1}
        \sum_{j\ne i} 
        \norm{L_{ii}\ind_ih_n - L_{ij}\ind_jh_n - M(L)\ind_ih_n}
        = 0.
    \end{align*}

    Next, we consider the block components in \eqref{eq: 1i-1j approx eigvecs} with output in $L^2(X_l)$ for $l\ne i$.
    First, putting $l=j$, we find that
        \[\lim_{n\to \infty} \norm{
            L_{li}(\ind_ih_n)
            - L_{ll}(\ind_lh_n)
            +M(L)\ind_l{h_n}}=0.\]
    Next, for $l\neq j$, we get
        \[\lim_{n\to \infty} \norm{
            L_{li}(\ind_ih_n)
            - L_{lj}(\ind_jh_n)}=0.\]
    Combining these and using the triangle inequality, we get that for all $l\ne i$ and $j\neq l$:
        \[\lim_{n\to \infty}\norm{L_{ll} \ind_l h_n - L_{lj}\ind_jh_n - M(L) \ind_l h_n}=0.\]
    Therefore,
    \begin{align*}
        &\lim_{n\to \infty}\norm{\left(D + \frac 1{k-1}T\right)h_n-M(L)h_n}_{X_l}\\[4pt]
        &=\lim_{n\to \infty}
        \norm{L_{ll} \ind_l h_n - \frac 1{k-1} \sum_{j\ne l} L_{lj} \ind_j h_n -M(L) \ind_l h_n}\\[4pt]
        &\le\lim_{n\to \infty}
        \frac 1{k-1}
        \sum_{j\ne l} 
        \norm{L_{ll} \ind_l h_n - L_{lj}\ind_jh_n - M(L) \ind_l h_n}
        =0
    \end{align*}
    for any $l\ne i$, as desired.
\end{proof}

We remark that a somewhat more general result in fact holds; see Subsection~\ref{subsec: further}.

\subsection{Proof of Theorem~\ref{thm: approx m-nik}}
\label{subsec: theorem a}

In this section, we turn to the proof of our main Theorem~\ref{thm: approx m-nik}. First, we deduce the following useful lemma from our technical Theorem~\ref{thm: meas Kol}.

\begin{lem}\label{cor: m-nik 1}
    Let \((X,\mu)\) be a measure space and $T : L^2(X) \to L^2(X)$ be a bounded self-adjoint operator. 
    Suppose there exists a measurable partition \(X=\bigsqcup _{i=1}^k X_i\) such that the block decomposition of $T$ with respect to $L^2(X)=\bigoplus_{i=1}^kL^2(X_i)$ has $0$ in every diagonal entry.
    Then, for any real bounded multiplication operator \(D\colon L^2(X)\to L^2(X)\),
    \[M(D-T) \ge M\!\left(D+\frac 1{k-1} T\right).\]
\end{lem}

\begin{proof}
    Set $L = D - T$. Note that, since the block decomposition of $T$ with respect to $\bigoplus_{i = 1}^k L^2(X_i)$ has $0$ in each diagonal entry, the block diagonal part of the block decomposition of $L$ with respect to $\bigoplus_{i = 1}^k L^2(X_i)$ is exactly $D$. Hence, by Theorem~\ref{thm: meas Kol}\ref{line: meas Kol 1}, taking $\xi = -\frac{1}{k-1}$, we obtain
    $$M(D-T) \geq M\!\left(D + \frac{1}{k + 1}T\right),$$
    as desired.
\end{proof}

The following observation regarding positive operators is also used in the proof.

\begin{lem} \label{lem: m-nik 2}
    Let $D$ and $T$ be bounded self-adjoint operators such that $D-T$ is positive. Then, for any $k>1$,
        \[M\!\left( D + \frac 1{k-1} T \right)
        \ge \frac k{k-1} M(T).\]
\end{lem}

\begin{proof}
    Fix $\eps>0$. There exists a unit vector $f\in H$ such that 
    \(\ang{T f, f}\ge M(T)-\eps\). Then,
    \begin{align*}
        M\!\left(D+\frac 1{k-1}T \right)
        & \ge \ang{\left(D + \frac 1{k-1} T\right)f, f} \\[4pt]
        & =\ang{\left(D-T \right) f,f}
            +\frac {k}{k-1}\ang{T 
            f,f}\\[4pt]
        & \ge \frac {k}{k-1} \ang{T  f,f}\\
        & \ge\frac {k}{k-1} \left(M(T)-\eps\right),
    \end{align*}
    where the penultimate inequality derives from the positivity of $D-T$.
    Taking the limit $\eps\to 0$, we get the desired inequality.
\end{proof}

The following corollary is an easy consequence of the fact that the Laplacian operator $L_\G = D_\G - T_\G$ is positive.

\begin{cor} \label{cor: m-nik 2 graphs}
    Let $\G$ be a bounded-degree Borel pmp graph on \((X,\mu)\). Then, for any $k>1$:
    \[
    M\!\left(D_\G + \frac 1{k-1} T_\G \right)
    \ge \frac k{k-1} M(T_\G).
    \]
\end{cor}

We can now turn to the proof of Theorem~\ref{thm: approx m-nik}.

\begin{proof}[Proof of Theorem~\ref{thm: approx m-nik}]
    Let $k = \chi_\mu^{\ap}(\G)$. Then for any \(\eps>0\), there exists a measurable subset $B^\eps\sub X$ with $\mu(B^\eps)<\eps$ and a measurable $k$-coloring \(c_\eps\colon X\backslash B^\eps\to \{1,\dots k\}\). Let \(A^\eps=X\backslash B^\eps\). Then, consider the block decompositions of $T_\G, D_\G$, and \(L_\G\) with respect to the direct sum 
    \(L^2(X)= L^2(A^\eps) \oplus L^2(B^\eps)\):
    
        \[T_\G =
        \begin{bmatrix}
        T^\eps_{A} & T^\eps_{12}\\[4pt]
        T^\eps_{21} & T^\eps_{B}
        \end{bmatrix},\
        D_\G =
        \begin{bmatrix}
        D^\eps_{A} & D^\eps_{12}\\[4pt]
        D^\eps_{21} & D^\eps_{B}
        \end{bmatrix},\
        L_\G =
        \begin{bmatrix}
        L^\eps_{A} & L^\eps_{12}\\[4pt]
        L^\eps_{21} & L^\eps_{B}
        \end{bmatrix}.\]
    By Lemma~\ref{lem: zero-diagonal}, the diagonal blocks of the block decomposition of $T_A^\eps$ with respect to 
    \(L^2(A^\eps)=\bigoplus_{i=1}^k L^2(X_i^\eps)\), where \(X_i^\eps={c_\eps}^{-1}(i)\) are the color classes of $c_\eps$, are zero. 
    Furthermore, note that $D_A^\eps$ is a bounded multiplication operator, and $D_A^\eps-T_A^\eps = L_A^\eps$ is positive, as it is a diagonal block component of a positive operator. 
    Hence, applying Lemma~\ref{cor: m-nik 1} and Lemma~\ref{lem: m-nik 2}, we get that for any \(\eps>0\),
        \begin{equation}\label{eq:Aineqs}
        M(D_A^\eps - T_A^\eps) 
        \ge M\!\left(D_A^\eps + \frac 1{k-1} T_A^\eps\right) 
        \ge \frac k{k-1} M(T_A^\eps)
        \end{equation}
    and thus
        \[0\ge k M(T_A^\eps) 
        - (k-1) M(L_A^\eps).\]
    Applying Lemma~\ref{lem: approx M} to $T_\G$ and $L_\G$, we get
        \[k M(T_\G) - (k-1) M(L_\G)
        = \lim_{\eps\to 0}
            \left( k M(T_A^\eps) 
            - (k-1) M(L_A^\eps\right))
        \le 0.\]
    Since $k$ is an integer, we conclude that
        \[ k\ge \ceil{\frac 
        {M(L_\G)}
        {M(L_\G) -M(T_\G)}} = \ceil{1+\frac 
        {M(T_\G)}
        {M(L_\G) -M(T_\G)}},\]
    as desired.
\end{proof}

The next result is immediate from the fact that $\chi_\mu(\G) \geq \chi_{\mu}^{\ap}(\G)$.

\begin{cor}\label{cor: meas nik}
    In the setting of Theorem~\ref{thm: approx m-nik}, 
    $$\chi_{\mu}(\G) \geq \ceil{1+\frac 
        {M(T_\G)}
        {M(L_\G) -M(T_\G)}}.$$
\end{cor}

\subsection{Characterizations of Equality}
Using Theorem~\ref{thm: meas Kol}, we can in fact characterize when many of the inequalities in the results of Subsection~\ref{subsec: theorem a} are equalities.

We begin by giving a necessary and sufficient condition for which equality holds in Corollary~\ref{cor: m-nik 2 graphs}. 

\begin{lem} \label{lem: m-nik 2 equality}
    Let $\G$ be a bounded-degree Borel pmp graph on \((X,\mu)\). If $\G$ is regular, then in the inequality
    \begin{equation}\label{eq: m-nik 2 equality case}
    M\!\left(D_\G + \frac 1{k-1} T_\G \right) \ge \frac k{k-1} M(T_\G)
    \end{equation}
    from Corollary~\ref{cor: m-nik 2 graphs}, equality holds for any $k>1$.
    Conversely, if we assume $\G$ is ergodic and \(M(T_\G)\) is in the discrete spectrum of $T_\G$, then for any $k>0$, equality in \eqref{eq: m-nik 2 equality case} implies $\G$ is $M(T_\G)$-regular.
\end{lem}

\begin{proof}
    If $\G$ is $d$-regular, then $M(T_\G)=d$. Moreover, the degree operator is simply given by $D_\G = d\cdot I$ so the spectrum of \(D_\G+\frac 1{k-1}T_\G\) is simply the spectrum of $\frac 1{k-1}T_\G$ shifted by $d$. So for any $k>1$ we have
    \[M\!\left(d\cdot I+\frac 1{k-1}T_\G \right)
    =d+\frac 1{k-1} MT_\G)
    =d+\frac 1{k-1} d
    =\frac k{k-1}d\]
    as desired.
    
    Conversely, suppose we have equality in \eqref{eq: m-nik 2 equality case} for some $k>1$. Further assume that $\G$ is ergodic and \(M(T_\G)\) is in the discrete spectrum of $T_\G$, i.e., there is some unit eigenvector \(f\in L^2(X)\) of $T_\G$ associated with \(M(T_\G)\). Then:
    \begin{align*}
        M\!\left(
        D_\G + \frac 1{k-1}T_\G 
        \right)
        &\ge
        \ang{\left(
        D_\G + \frac 1{k-1}T_\G 
        \right)f,f}\\[4pt]
        & =\ang{\left(D_\G - T_\G\right) f,f}
            +\frac {k}{k-1}\ang{T_\G  f,f}\\[4pt]
        &\ge \frac {k}{k-1}\ang{T_\G  f,f}
        = \frac {k}{k-1}M(T_\G).
    \end{align*}
    Since we have equality in \eqref{eq: m-nik 2 equality case}, this forces \(\ang{L_\G f,f} = \ang{\left(D_\G - T_\G \right)f,f}=0\).
    Since $\G$ is ergodic, using Lemma~\ref{lem: ergodic lap. quadratic form = 0}, we see that this implies that $f$ is constant on some invariant conull set \(X_f\sub X\).
    Hence, without loss of generality, $f=1$ on \(X_f\), which implies that for all \(x\in X_f\), we have \((T_\G  f)(x) =M(T_\G) f(x)\). Therefore
    \[
        \sum_{y\sim x} f(y)
        =M(T_\G).         
    \] 
    Recall that \(X_f\) is invariant, so any $y$ adjacent to $x$ is also in \(X_f\), implying \(f(y)=1\) as well, so
    \[\deg(x)=M(T_\G),\]
    i.e., $\G$ is $d$-regular on \(X_f\), with \(d=M(T_\G)\).
\end{proof}

Next, we turn to the equality cases for Corollaries~\ref{cor: m-nik 1} and~\ref{cor: meas nik}. First, it is immediate from Theorem~\ref{thm: meas Kol}\ref{line: meas Kol 3} that the inequality in Corollary~\ref{cor: m-nik 1} is an equality, i.e.,
    \[M(L)=M\!\left(D+\frac 1{k-1} T\right),\]
if and only if there are some sequence $(h_n)_{n\in \N}$ in $L^2(X)$ and some \(i\in \en{k}\) such that
    \begin{equation}
    \lim_{n\to \infty}
    \frac{\ang{L
        (\left(\ind_i-\ind_j\right) h_n),
        \left(\ind_i-\ind_j\right) h_n
        }}
        {\norm{\left(\ind_i-\ind_j\right) h_n}^2}
    =M(L) \label{eq: equality condition}
    \end{equation}
for each $j\ne i$.

Finally, we observe that this equality criterion, in conjunction with Lemma~\ref{lem: m-nik 2 equality}, characterizes when equality holds in the inequality
\begin{equation}\label{eq:chimuineq}
\chi_{\mu}(\G) \geq \ceil{1+\frac 
        {M(T_\G)}
        {M(L_\G) -M(T_\G)}}
\end{equation}
from Corollary~\ref{cor: meas nik}.

\begin{cor}\label{cor: meas char}
    Suppose $\G$ is a bounded-degree Borel pmp graph on $(X,\mu)$ with measurable chromatic number $k=\chi_\mu(\G)$ and Laplacian operator $L=L_\G$. 
    
    If $\G$ is regular and there exist some sequence $(h_n)_{n\in \N}$ in $L^2(X)$ and some $i\in \en{k}$ which satisfy \eqref{eq: equality condition} for all $j\ne i$, then equality holds in \eqref{eq:chimuineq}.

    Conversely, assume $\G$ is ergodic and $M(T_\G)$ is in the discrete spectrum of $T_\G$. Then if equality holds in \eqref{eq:chimuineq}, then $\G$ is $M(T_\G)$-regular and there exist some sequence $(h_n)_{n\in \N}$ in $L^2(X)$ and some $i\in \en{k}$ which satisfy \eqref{eq: equality condition} for all $j\ne i$.
\end{cor}
\begin{proof}
    Let $c\colon X\to \{1,\dots,k\}$ be a measurable coloring of $\G$ with color classes $X_i=c^{-1}(i)$. Let $\xi=-\frac 1{k-1}$. By Lemma~\ref{lem: zero-diagonal}, the diagonal block operator \(D = Diag(L_{11},\dots,L_{kk})\) of $L_\G$ is simply the degree operator $D_\G$. We are thus in a position to apply Theorem~\ref{thm: meas Kol} to $L_\G$, $D_\G$, and $T_\G$.
    
    First, if $\G$ is regular and there exist some sequence $(h_n)_{n\in \N}$ in $L^2(X)$ and some $i\in \en{k}$ which satisfy \eqref{eq: equality condition} for all $j\ne i$, then it follows immediately from Lemma~\ref{lem: m-nik 2 equality} and Theorem~\ref{thm: meas Kol}\ref{line: meas Kol 3}, that
    \[
    M\!\left(D_\G + \frac 1{k-1} T_\G \right) = \frac k{k-1} M(T_\G),
    \]
    and
    \[
    M(L_\G) = M\!\left(D_\G + \frac 1{k-1} T_\G\right).
    \]
    Combining these, we immediately get
    \[\chi_\mu(\G) = k = \ceil {
        1+\frac{M(T_\G)}{M(L_\G)-M(T_\G)}
    }.\]
    
    Conversely, assume we have the equality for an ergodic $\G$ with $M(T_\G)$ in the discrete spectrum of $T_\G$. Then all inequalities in equation~\eqref{eq:Aineqs} (without $\eps$'s) in the proof of Theorem~\ref{thm: approx m-nik} are equalities. In particular, it follows that 
    \[
    M\!\left(D_\G + \frac 1{k-1} T_\G \right) = \frac k{k-1} M(T_\G),
    \]
    and
    \[
    M(L_\G) = M\!\left(D_\G + \frac 1{k-1} T_\G\right).
    \]
    Thus, by Lemma~\ref{lem: m-nik 2 equality} and Theorem~\ref{thm: meas Kol}\ref{line: meas Kol 3}, we immediately get that $\G$ is $M(T_\G)$-regular and there exist some sequence $(h_n)_{n\in \N}$ in $L^2(X)$ and some $i\in \en{k}$ which satisfy \eqref{eq: equality condition} for all $j\ne i$. 
\end{proof}

Note that since $\chi_\mu^\ap(\G)\le \chi_\mu(\G)$, the conditions from Corollary~\ref{cor: meas char} are sufficient, but perhaps not necessary, to also show 
    \[\chi_\mu^\ap(\G)=
    \ceil{1+ \frac{M(T_\G)} {M(L_\G)-M(T_\G)}}.\]
    
\subsection{Further Results} 
\label{subsec: further}
As mentioned earlier, using Theorem~\ref{thm: meas Kol}\ref{line: meas Kol 1}, we can in fact obtain a more general result, which yields the same inequality as in Theorem~\ref{thm: meas Kol}\ref{line: meas Kol 1} but allows, for each $\eps > 0$, for the measurable $k$-partition of $X$ to disinclude an error set of measure less than $\eps$. Note that condition~(ii) below is somewhat restrictive, but it is in fact satisfied when $D$ is diagonal and $T$ has zero diagonal blocks with respect to the given partitions; in particular, this holds for $D_\G$ and $T_\G$ when the partitions come from approximate measurable colorings.

\begin{cor}\label{approx meas Kol}
    Let $L, D, T\colon L^2(X)\to L^2(X)$ be bounded self-adjoint operators with $T=D-L$. 
    Suppose that there exists some fixed $k$ such that for any $\eps>0$ there exists a measurable partition $X = \bigsqcup_{i=1}^{k+1} X_i$ for which the following hold:
    \begin{enumerate-(i)}
        \item $\mu(X_{k+1})<\eps$;
        \item $D$ is the block diagonal part of $L$ in the first $k$ blocks. That is, letting $A=\bigsqcup_{i=1}^k X_i$,
        \[D|_{A} = Diag\left(L_{11},\dots,L_{kk}\right),\]
        where \(L=(L_{ij})_{i,j=1}^{k+1}\) is the block decomposition of $L$ according to the direct sum \(L^2(X)=\bigoplus_{i=1}^{k+1} L^2(X_i)\).
    \end{enumerate-(i)}
    Then, for every \(\xi\in \left[-\frac 1{k-1}, 1\right],\)
            \[M(L)\ge M(D-\xi T ).\]
\end{cor}

\begin{proof}
    Fix $\eps>0$. Consider the partition $X=A \sqcup B$, where $A=\bigsqcup_{i=1}^k X_i$ and $B=X_{k+1}$.  
    Then, consider the block decompositions of $T,D,L$ with respect to the direct sum 
    \(L^2(X)= L^2(A) \oplus L^2(B)\):
        \[T = \begin{bmatrix}
        T^\eps_{A} & T^\eps_{12}\\[4pt]
        T^\eps_{21} & T^\eps_{B}
        \end{bmatrix},\
        D = \begin{bmatrix}
        D^\eps_{A} & D^\eps_{12}\\[4pt]
        D^\eps_{21} & D^\eps_{B}
        \end{bmatrix},\
        L = \begin{bmatrix}
        L^\eps_{A} & L^\eps_{12}\\[4pt]
        L^\eps_{21} & L^\eps_{B}
        \end{bmatrix}\]
    By assumption, $D^\eps_A$ is the block diagonal part of $L^\eps_A$. So we can apply Theorem~\ref{thm: meas Kol} to the upper blocks to get
        \[M(L^\eps_{A})
        \ge M(D^\eps_{A}-\xi T^\eps_{A}).\]
    Together with Proposition~\ref{prop: block decomp T}, we can conclude that
     \[M(L)\ge M(L^\eps_{A})
        \ge M(D^\eps_{A}-\xi T^\eps_{A}).\]
    Taking the limit as $\eps\to 0$ and using Lemma~\ref{lem: approx M}, we see that indeed
        \[M(L)\ge M(D-\xi T),\]
    as desired.
\end{proof}

We remark also that, from Corollary~\ref{cor: m-nik 1}, we obtain a novel proof of the approximate measurable version of Hoffman's theorem first shown in \cite[Theorem~6.8]{HST2026}.

\begin{cor}
    \cite[Theorem~6.8]{HST2026} Let $\G$ be a bounded-degree Borel pmp graph on $(X, \mu)$. Then 
    $$\chi_{\mu}^{\ap}(\G) \geq \left\lceil 1 - \frac{M(T_\G)}{m(T_\G)} \right\rceil.$$
\end{cor}

\begin{proof}
    Write $T = T_\G, L = L_\G$, $D = 0$, and $k' = \chi_{\mu}(\G)$, and let $X = \bigsqcup_{i = 1}^{k'} X_i$ be a partition of $X$ into measurable independent sets. Then by Corollary~\ref{cor: m-nik 1} and Lemma~\ref{lem: zero-diagonal},
    \[-m(T_\G)
    =M(-T_\G) 
    \ge \frac{1}{\chi_\mu(\G) - 1} M(T_\G).\]
It follows that
    \[\chi_\mu(\G)
    \ge \ceil{1 -\frac{M(T_\G)}{m(T_\G)}}.\]

Now fix $\eps>0$ and let $k=\chi_\mu^\ap(\G)$. Then there exist a measurable set $B^\eps$ with $\mu(B^\eps)<\eps$ and a measurable $k$-coloring \(c_\eps\colon X\backslash B^\eps\to \{1,\dots k\}\). Let \(A^\eps=X\backslash B^\eps\). Then, consider the block decomposition of $T_\G$ with respect to the direct sum 
\(L^2(X)= L^2(A^\eps) \oplus L^2(B^\eps)\):    
    \[T_\G =
    \begin{bmatrix}
    T^\eps_{A} & T^\eps_{12}\\[4pt]
    T^\eps_{21} & T^\eps_{B}
    \end{bmatrix}.\]
Since the induced subgraph $\G[A^\eps]$ has chromatic number at most $k$, as in the first part of the proof, we get
    \[-m(T^\eps_A)
    \ge \frac1{k - 1} M(T^\eps_A).\]
Then, by Proposition~\ref{prop: block decomp T}\ref{line: block decomp Mm}, we have \(-m(T_\G) \ge \frac1{k - 1} M(T^\eps_A)\), so taking the limit as $\eps\to 0$ and applying Lemma~\ref{lem: approx M}, we have
    \[-m(T_\G)
    \ge \frac{1}{k - 1} M(T_\G).\]
Since $k=\chi^\ap_\mu(\G)$, this gives the desired inequality.
\end{proof}

\section{Bipartiteness in the non-regular case}\label{sec: bipartiteness}

In this section, we turn to the question of approximate measurable bipartiteness, i.e., the situation when $\chi_\mu^\ap(\G)=2$. For a finite connected graph $G$, it is well-known that $G$ is bipartite if and only if the spectrum of its adjacency operator is symmetric about 0; in fact, for the converse direction, it is enough that the negative of the maximum eigenvalue of the adjacency matrix of $G$ is also an eigenvalue. We adapt this characterization to the approximate measurable setting.
This result draws inspiration from the proof approach in \cite[Theorem~4.9]{HST2026} and generalizes the result to the non-regular case.

\bipartiteness*

\begin{proof}
    Let $M=M(T_\G)$. Assume that \(-M\in \sigma(T_\G)\) and fix $1>\eps>0$. Our goal is to construct disjoint $\G$-independent Borel subsets $A,B\sub X$ with $\mu(A\cup B)>1-\eps$.
    
    Since $\G$ has spectral gap, there exists some unit eigenvector of $T_\G$ associated with $M$, which we will denote by $h\in L^2(X)$. By \cite[Lemma~3.18]{HST2026}, we can assume $h$ is strictly positive.
    
    Let $(f_n)_{n\in \N}$ be a sequence of unit functions in $L^2(X)$ such that \(\norm{T_\G f_n+Mf_n }\to0\). Note that we can assume that the $f_n$ are all real-valued. (If \(f_n = u_n + iv_n\) is the decomposition of $f_n$ into its real and imaginary parts, then the decomposition of $T_\G f_n$ into its real and imaginary parts is simply given by $T_\G f_n = T_\G u_n + iT_\G v_n$.)

    \begin{claim}\label{claim: abs f_ns are multiples of h}
        Without loss of generality, we can assume $\abs{f_n} $ is a multiple of $h$ for all $n\in \N$.
    \end{claim}

    \begin{proof}[Proof of Claim~\ref{claim: abs f_ns are multiples of h}]
        By Lemma~\ref{lem: -M and abs eig'fns}, 
        \(\norm{T_\G \abs{f_n} - M\abs{f_n}} \to 0\), i.e., $(g_n)_{n\in \N}=(\abs{f_n})_{n\in \N}$ is a sequence of approximate eigenvectors for $M$.
        So we can break down $g_n$ into its orthogonal decomposition with respect to $\C h$ and $(\C h)^\perp$:
            \[g_n = g_n^{(0)} + g_n^{(1)},\quad 
            g_n^{(0)} = \ang{g_n, h} h \]
        Then, by \cite[Lemma~3.18]{HST2026} and \cite[Proposition~2.3]{HST2026}, together with the assumption of spectral gap, since $g_n^{(1)}\in (\C h)^\perp$ there exists $c>0$ such that
            \[\norm{T_\G g_n^{(1)} -Mg_n^{(1)}} 
            \ge c\norm{g_n^{(1)}}\]
        for any $n\in \N$.

        Since $g_n^{(0)}$ is a scalar multiple of $h$, it must also be an eigenvector, i.e., $T_\G  g_n^{(0)}  = Mg_n^{(0)}$. This implies
            \[\norm{T_\G g_n^{(1)} - Mg_n^{(1)} } 
            = \norm{T_\G g_n -Mg_n} \to 0,\]
        forcing $\norm{g_n^{(1)}} \to 0$. 
        Consider now the function defined by
            \[\tilde f_n(x) = f_n(x) - \sgn(f_n(x)) \cdot g_n^{(1)}(x).\]
        Then observe that $\abs{\tilde f_n}=\abs{g_n^{(0)}}$:
            \[\abs{\tilde f_n(x)} =
                \begin{cases} 
                \abs{f_n(x)-g_n^{(1)}(x) } 
                =\abs{\abs{f_n(x)}-g_n^{(1)}(x)} 
                =\abs{g_n^{(0)}(x)} 
                &\text{if }f_n(x)\ge 0,\\[4pt]
                \abs{f_n(x)+g_n^{(1)}(x)}
                =\abs{-\norm{f_n(x)}+g_n^{(1)}(x)} 
                =\abs{g_n^{(0)}(x)} 
                &\text{if }f_n(x)< 0.\\[4pt]  
                \end{cases}\]
        Thus, $\abs{\tilde f_n}$ is also a multiple of $h$ (as $h$ is strictly positive). Moreover,
        \begin{align*} 
            \norm{T_\G \tilde f_n+M\tilde f_n} 
            &\le \norm{T_\G f_n+Mf_n} + \norm{T_\G g_n^{(1)}} + M\norm{g_n^{(1)}}\\
            &\le \norm{T_\G f_n+Mf_n} + 2M\norm{g_n^{(1)}}\to 0, 
        \end{align*}
        as desired.
    \end{proof}
    
    So, rescaling if necessary, we can find a sequence $(f_n)_{n\in \N}$ of unit functions in $L^2(X)$ such that $\abs{f_n}=h$ for any $n\in \N$ and $\norm{T_\G f_n+Mf_n}\to0$.
    Since $h(x)$ is strictly positive, we can pick a sufficiently small $\delta>0$ such that the set \(Z=\st{x\in X :h(x)\le \delta }\) satisfies $\mu(Z)<\frac \eps 2$. 
    Then, for all $n\in \N$ let
        \[Y_n=\st{ x\in X: \abs{T_\G f_n(x)+Mf_n(x)} \ge \delta}.\]
    Notice that
        \[\mu(Y_n) \cdot \delta^2
        =\int_{Y_n} \delta^2 \dmu
        \leq\int_X{|T_\G f_n+Mf_n|^2} \dmu
        = \norm{T_\G f_n+Mf_n}^2 \to 0.\]
    This implies there exists $N\in \N$ such that $\mu(Y_N)<\frac \eps 2$. Then $X'=X\backslash(Y_N\cup Z)$ has measure $\mu(X')>1-\eps$. Recall that $\abs{f_N}=h$ so we can partition $X'$ into the sets
        \[A=\st{ x\in X': f_N(x)=h(x) },\ 
        B=\st{ x\in X': f_N(x)=-h(x) }.\]
    It remains to show that $A$ and $B$ are each $\G$-independent. Let $x\in A$ and observe that since $h$ is an eigenvector of $T_\G$ associated with $M$, we have
        \[T_\G f_N(x) + Mf_N(x)  
        =T_\G f_N(x) + T_\G h(x)  
        =\sum_{y\sim x} \left(f_N(y)+h(y)\right).\]
    Since $f_N(y)=\pm h(y)$ for all $y\in X$, each term in this sum is either $0$ or $2h(y)$, i.e., $  T_\G f_N(x)+Mf_N(x)$ is nonnegative.
    Assume for the sake of contradiction there is some neighbor $y_0$ of $x$ such that $y_0\in A$. Then
        \[\abs{T_\G f_N(x)+Mf_N(x)}
        =\sum_{y\sim x} \left(f_N(y)+h(y)\right) 
        \ge 2h(y_0) >2\delta,\]
    since $y_0\notin Z$.
    However, also \(x\notin Y_N\), i.e., \(\abs{T_\G f_N(x)+Mf_N(x)} <\delta\), a contradiction.
    By a symmetric argument, $B$ must also be $\G$-independent, as desired.
\end{proof}

We remark here that, while the ergodicity assumption in Theorem~\ref{thm: bipartite} is necessary (see \cite[Remark~4.13]{HST2026}), we do not know whether the spectral gap assumption is also necessary.

\section{The Measurable Domination Number and the Laplacian Spectral Gap}
\label{sec: domination}

In this section, we define the measurable domination number for a bounded-degree Borel pmp graph and use it to give bounds on the size of the Laplacian spectral gap. Several of our results are inspired by certain classical theorems of \cite{LLT05}.

We say a measurable subset $Y\sub X$ is a \emph{dominating set} of a bounded-degree Borel pmp graph $\G$ on $(X, \mu)$ if $Y \cup \mathcal{N}(Y) = X$ (up to measure zero).
The \emph{measurable domination number} $\gmu{\G}$ of $\G$ is then given by
    \[\gmu{\G}=\inf_{\substack{Y\sub X,\\Y\text{ dominating}}} \mu(Y).\]

First, using Lemma~\ref{lem: edge density representor}, we can easily generalize \cite[Theorem~10]{LLT05}, though note that Proposition~\ref{prop: generalized gamma bound} below provides a sharper bound.

\begin{prop}\label{prop: meas llt 10}
    Let $\G$ be a bounded-degree Borel pmp graph on a standard probability space $(X,\mu)$. Then,
    \[\gmu{\G}\ge \frac 1{M(L_\G)}.\]
\end{prop}

\begin{proof}
    Let $Y\sub X$ be a measurable dominating set of $\G $, and let $E_Y = E(\G)\cap \left(Y\times Y^c\right)$ be the set of its boundary edges. An easy computation then gives
    \[
        \nu(E_Y) = \int_{Y^c} \deg_{Y}(x)\ d\mu(x) \ge \int_{Y^c} 1\ d\mu(x) =\mu(Y^c)=1-\mu(Y)
    \]
    where the inequality follows from the definition of a dominating set.
    By Lemma~\ref{lem: edge density representor}, there exists a vector $f\in L^2(X)$ such that
        \[M(L_\G) 
        \ge \frac{\ang{L_\G f, f}} {\norm{f}^2}\
        = \frac {\nu(E_Y)}{\mu(Y)(1-\mu(Y))} 
        \ge \frac 1{\mu(Y)},\]
    and the result follows.
\end{proof}

The following result gives an alternative lower bound on $\gmu{\G}$ in terms of the maximum degree. Note that it follows from Lemma~\ref{lem: M(L) degree bound} that this bound improves on the one from Proposition~\ref{prop: meas llt 10}.

\begin{prop}\label{prop: generalized gamma bound}
    Let $\G$ be a bounded-degree Borel pmp graph on a standard probability space $(X,\mu)$ with maximum degree $\Delta$. Then,
    \[\gmu{\G}\ge \frac 1{\Delta+1}.\]
\end{prop}
\begin{proof}
    Assume $A\subset X$ is a dominating set. By definition, this implies that $\mu(A\cup\calN(A)) = 1$. Since $\G$ has maximum degree $\Delta$, it follows from \cite[Corollary~2.11]{HST2026} that $\mu(\calN(A))\leq \Delta\mu(A)$. We thus conclude that
    \[
    \mu(A) + \Delta\mu(A) \geq \mu(A) + \mu(\calN(A)) \geq \mu(A\cup\calN(A)) = 1.
    \]
    Hence $\mu(A) \geq \frac1{\Delta+1}$, and the result follows.
\end{proof}

We now turn towards our main theorem of this section, Theorem~\ref{thm: domination}. 
We start with the following proposition, inspired by \cite[Lemma~2]{LLT05}, which guarantees the existence of a dominating set $Y$ with measure arbitrarily close to $\gmu{\G}$ and such that ``most'' $x\in Y$ have a \emph{sole neighbor}, i.e., a vertex $y\in Y^c$ whose only neighbor in $Y$ is $x$. For each $x \in Y$, we will write $W_x(Y)$ for its set of sole neighbors, i.e.,
$$W_x(Y) = \{y \in Y^c : \calN(y) \cap Y = \{x\} \}.$$
When $Y$ is understood, we also write $W_x$ for $W_x(Y)$.

\begin{prop}\label{prop: llt lemma 2}
     Let $\G$ be a bounded-degree Borel pmp graph with no isolated points (up to measure zero). Then, for any $\eps>0$, there exists a Borel dominating set $Y$ of $\G$ with $\mu(Y) < \gmu{\G}+\eps$ such that
     \[\mu( \left\{ x\in Y: W_x(Y) \ne\empt \right\}) >(1-\eps) \mu(Y).\]
\end{prop}
\begin{proof}
We will write $\gamma_\mu = \gmu{\G}$ for convenience. Fix $\eps>0$, and let $\eps' = \frac{\eps}{2\Delta^2} \gamma_\mu$, where $\Delta$ denotes the maximum degree of $\G$. 
Let $a\colon \mathcal B(X)\to \R$ denote the function which assigns the average internal degree to any measurable subset of $X$:  
\[a(Y) = \frac 1{\mu(Y)} \int_{Y} \deg_Y(x) \dmu(x).\]

Consider some dominating set $Y$ with $\mu(Y)<\gamma_\mu+\eps'$ such that 
\begin{equation}\label{eq: high ave degree} a(Y)>\sup_{\substack{D\sub X,\\ D\text{ dominating},\\ \mu(D)<\gamma_\mu+\eps'} } \left( a(D)\right) -\delta \end{equation}  
where $\delta = \dfrac \eps{3\Delta^2}$.  
Let $A$ and $B$ be the sets of ``bad'' vertices $x$, i.e., for which $W_x = \empt$, separated based on whether their $Y$-degree is zero or not:  
\begin{align*} 
A = \left\{ x \in Y: W_x = \empt,\quad \deg_Y(x)>0 \right\},\\  
B = \left\{ x \in Y: W_x = \empt,\quad \deg_Y(x)=0 \right\}. 
\end{align*}

It is clear that $A$ and $B$ are both measurable. We will show $\mu(A)+\mu(B)<\eps\mu (Y)$.

\begin{claim}\label{claim: redundant vertices}
$\mu(A)<\frac \eps 2 \mu(Y)$.
\end{claim}

\begin{proof}[Proof of Claim~\ref{claim: redundant vertices}]
    
    Assume for the sake of contradiction that $\mu(A) >\frac \eps2 \mu(Y)$.

    Consider the distance-2 graph ${\G }^2$; since $\G $ has a degree bound of $\Delta$, ${\G }^2$ must have a degree bound of $\Delta(\Delta-1)=\Delta^2-\Delta$. Therefore, ${\G }^2$ can be colored with $k$ Borel color classes $X=\bigsqcup_{i=1}^{k} X_i$ where $k=\Delta^2-\Delta+1$. So, we may partition $A$ into $k$ Borel sets $(A_i)_{i=1}^k$ where $A_i=A\cap X_i$. Since the neighborhoods of any two distinct vertices in $A_1$ must be disjoint, every element of $A_1$ has a neighbor in $Y\backslash A_1$ and no sole neighbors outside $Y$. Thus, we can define a new Borel dominating set $Y^\star=Y\backslash A_1$. 
    
    By the pigeonhole principle, without loss of generality, 
    \[\mu(A_1)\ge  \frac{\mu(A)}{k} > \frac{\eps}{2k} \mu(Y).\]  
    This implies
    \[\mu(Y)
    =\mu(Y^\star)+\mu(A_1) 
    > \mu(Y^\star)+  \frac{\eps}{2k} \mu(Y)
    \geq \gamma_\mu +  \frac{\eps}{2\Delta^2} \mu(Y).\] 
    However, this is a contradiction with the assumption that $\mu(Y)<\gamma_\mu+\eps'$. So indeed, $\mu(A)<\frac \eps 2 \mu(Y)$.
\end{proof}

\begin{claim}\label{claim: inefficient vertices}
$\mu(B)<\frac \eps 2 \mu(Y)$.
\end{claim}

\begin{proof}[Proof of Claim~\ref{claim: inefficient vertices}]
Assume for the sake of contradiction that $\mu(B) >\frac \eps2 \mu(Y)$.  It suffices to replace some of these vertices in $B$, creating a new dominating set $Y'$ with the same measure as $Y$ but $a(Y')> a(Y)+\delta$, thus deriving a contradiction with the assumption \eqref{eq: high ave degree}.  

Let $f\colon B\to X$ be a Borel function such that $f(x)\in \calN (x)$ for all $x\in B$ (for instance, by sending $x$ to its least neighbor according to some Borel linear order on $X$). 
Recall the partition $X=\bigsqcup_{i=1}^k X_i$ of $X$ into Borel color classes of $\G^2$ from the proof of Claim~\ref{claim: redundant vertices}.
We may partition $B$ into $k$ Borel sets $(B_i)_{i=1}^k$ where $B_i=B\cap X_i$. By the pigeonhole principle, without loss of generality, we can again assume that
    \[\mu(B_1)
    \ge  \frac{\mu(B)}{k} 
    > \frac{\eps}{2k} \mu(Y). \]  
Since the neighborhoods of any two distinct vertices in $B_1$ must be disjoint, $f$ is injective on $B_1$. In particular, $f|_{B_1}$ is a partial Borel bijection that preserves the connectedness relation. Since $\G$ is pmp, we thus conclude that
\begin{equation}\label{eq:mufB1}
\mu(f(B_1))=\mu(B_1)> \frac{\eps}{2k} \mu(Y).
\end{equation}
Now, consider $Y'=Y\backslash B_1 \sqcup f(B_1)$. Since $B_1$ and $f(B_1)$ are disjoint by construction, we get $\mu(Y')=\mu(Y)<\gamma_\mu+\eps'$. Moreover, since $W_x = \empt$ for every $x\in B$, $Y'$ is a dominating set.

It remains to show $a(Y')> a(Y)+\delta$. First, recall that $\deg_Y(x)=0$ for all $x\in B_1$, and thus $\deg_{Y'}(x)\ge \deg _Y(x)$ for any $x\in Y\backslash B_1$. 
This implies  
\begin{align}
\begin{split}\label{eq:aY}
    a(Y)
    &= \frac 1{\mu(Y)} \left[ \int_{Y\backslash B_1} \deg_{Y}(x) \ d\mu(x) + \int_{B_1} \deg_{Y}(x) \ d\mu(x) \right]\\
    &\le \frac 1{\mu(Y)} \int_{Y\backslash B_1} \deg_{Y'}(x) \ d\mu(x).
\end{split}
\end{align}
Further recall that vertices in $B_1$ have no sole neighbors and the neighborhoods of distinct vertices in $B_1$ are disjoint. So for every $x\in f(B_1)$, there must exist some $y\in Y\backslash B_1$ such that $x\sim y$. This implies $\deg_{Y'}(x)\ge1$ for any $x\in f(B_1)$. Combining this with \eqref{eq:mufB1} and \eqref{eq:aY}, we get
\begin{align*}
    a(Y') 
    &= \frac 1{\mu(Y')} \left[ \int_{Y\backslash B_1} \deg_{Y'}(x) \ d\mu(x)  + \int_{f(B_1)} \deg_{Y'}(x) \ d\mu(x)  \right]\\  
    &\ge \frac 1{\mu(Y)} \left[ \int_{Y\backslash B_1} \deg_{Y'}(x) \ d\mu(x)  + \mu(f(B_1))  \right]\\
    &\ge  a(Y) + \frac{\eps}{2k} \\
    &> a(Y)+\frac{\eps}{3\Delta^2}.
\end{align*}
Thus $a(Y')>a(Y)+\delta$, a contradiction with \eqref{eq: high ave degree}, as desired.
\end{proof}
Combining Claims~\ref{claim: redundant vertices} and \ref{claim: inefficient vertices}, we deduce that $\mu( \left\{ x\in Y: W_x \ne\empt \right\}) >(1-\eps) \mu(Y)$. Since also $\mu(Y) < \gamma_\mu + \eps' < \gamma_\mu + \eps$, this finishes the proof.
\end{proof}

The main result of this section is inspired by \cite[Theorem~8]{LLT05} and uses Lemma~\ref{lem: edge density representor} and Proposition~\ref{prop: llt lemma 2}. We note that our bound is different from the bound in \cite[Theorem~8]{LLT05}, even for finite graphs, and that they are in general incomparable. The reason for this is that, while part of the proof is similar, \cite[Theorem~8]{LLT05} bounds the degrees of certain vertices essentially by the total number of vertices, while our proof uses the maximum degree (which therefore shows up in the inequality below).

\domination*

\begin{proof}
    Fix $\eps>0$. 
    Applying Proposition~\ref{prop: llt lemma 2}, we can find a dominating set $Y$ of $\G$ such that $\mu(Y)<\gamma_\mu+\eps$ and
    \[\mu(\st{x\in Y : W_x\ne \empt})>(1-\eps)\mu(Y),\] where
    \[W_x=\st{
        w\in Y^c : \calN(w)\cap Y = \{x\}
    }.\]
    Let $Y': = \st{x\in Y : W_x\ne \empt}$ be the set consisting of vertices that have sole neighbors, and $W=\bigcup_{x\in Y'}W_x$ consist of the sole neighbors.
    Then, let $F\colon X\times X\to \{0,1\}$ be the Borel transport defined by
    \[F(x,y)=\begin{cases}
        1&\text{if }x\in Y'\text{ and } y\in W_x\\
        0&\text{otherwise}
    \end{cases}.
    \] 
    Note that, for any $x \in Y'$, $\out_F(x) = \sum_{y \in W_x} F(x, y) = \vert W_x \vert$, and for any $x \in W$, $\ins_F(x) = \sum_{y \in Y'} F(y, x) = 1$. Hence
    \[\out_F(x) =\begin{cases}
        \vert W_x \vert &\text{if }x\in Y',\\
        0&\text{otherwise},
    \end{cases}
    \]
    and
    \[\ins_F(x)=\begin{cases}
        1&\text{if }x\in W,\\
        0&\text{otherwise}.
    \end{cases}
    \]
    By the mass transport principle~\eqref{eq:masstransport}, we thus obtain
    \begin{align*}
        \mu(W) & = \int_W 1\dmu(x) = \int_X \ins_F(x)\dmu(x) \\
        & = \int_X \out_F(x) \dmu(x) = \int_{Y'} \vert W_x \vert \dmu(x) \\
        & \geq \mu(Y') > (1 - \eps)\mu(Y).
    \end{align*}
    Letting $E_Y=E(\G)\cap (Y\times Y^c)$ be the set of boundary edges of $Y$, we can now compute 
    \begin{align*}\nu(E_Y)
    & = \int_{W} \deg_Y(x) \dmu(x)
     + \int_{Y^c\backslash W} \deg_Y(x) \dmu(x)\\
    & \le \int_{W} 1 \dmu(x)
     + \int_{Y^c\backslash W} \Delta \dmu(x)\\
     & = \mu(W)
     + \Delta(1-\mu(Y)-\mu(W))\\
     & = \Delta(1-\mu(Y))
     - (\Delta-1)(\mu(W))\\
     &<\Delta(1-\mu(Y))
     - (\Delta-1)(1-\eps) \mu(Y).
    \end{align*}
    Then, by Lemma~\ref{lem: edge density representor}, there exists some $f\in L^2_0(X)$ such that
     \[\frac{\ang{L_\G f, f}} {\norm{f}^2}\
        = \frac {\nu(E_Y)}{\mu(Y)(1-\mu(Y))}.\]
    This implies
    \[m_0(L_\G)
    <
    \frac
    {\Delta(1-\mu(Y)) - (\Delta-1)(1-\eps) \mu(Y)}
     {\mu(Y)(1-\mu(Y))}.
    \] 
    Then, taking the limit as $\eps\to 0$, 
    \[m_0(L_\G)
    \le
    \frac
    {\Delta(1-\gamma_\mu) - (\Delta-1) \gamma_\mu}
     {\gamma_\mu(1-\gamma_\mu)}
    =
    \frac
    {\Delta(1-2\gamma_\mu) + \gamma_\mu}
     {\gamma_\mu(1-\gamma_\mu)},
    \] 
    as desired.
\end{proof}

\bibliographystyle{amsalpha}
\bibliography{paper}

\end{document}